\documentclass{amsart}
\usepackage{amsfonts}
\usepackage{color}
\usepackage{graphicx}
\usepackage{mathrsfs}
\usepackage{enumerate}
\usepackage{color}
\usepackage[utf8]{inputenc}
\usepackage[bookmarksnumbered,colorlinks, linkcolor=blue, citecolor=red, pagebackref, bookmarks, breaklinks]{hyperref}

\newtheorem{thm}{Theorem}[section]

\newtheorem{lemma}[thm]{Lemma}
\newtheorem{prp}[thm]{Proposition}
\theoremstyle{definition}
\newtheorem{definition}[thm]{Definition}

\theoremstyle{remark}
\newtheorem{remark}[thm]{Remark}

\numberwithin{equation}{section}
\newcommand{\R}{\mathbb R}

\newcommand{\RR}{\mathbb{R}}

\newcommand{\ol}{\overline}
\newcommand{\Om}{\Omega}
\newcommand{\rarrow}{\rightarrow}
\newcommand{\e}{\epsilon}

\begin{document}

\title[Double phase Robin problem]{Fractional double phase Robin problem involving variable order-exponents without Ambrosetti-Rabinowitz condition}
\author[R. Biswas]{Reshmi Biswas}
\author[S. Bahrouni]{Sabri Bahrouni}
\author[M. L. Carvalho]{Marcos L. Carvalho}
\address[R. Biswas]{Mathematics Department, Indian Institute of technology Guwahati, Guwahati, Assam 781039, India}
\address[S. Bahrouni]{Mathematics Department, Faculty of Sciences, University of Monastir, 5019 Monastir, Tunisia}
\address[M. L. Carvalho]{Mathematics Institute, Universidade Federal de Goi\'as, Brazil}
\email[R. Biswas]{b.reshmi@iitg.ac.in}
\email[S. Bahrouni]{sabri.bahrouni@fsm.rnu.tn}
\email[M. L. Carvalho]{marcos\_leandro\_carvalho@ufg.br}

\keywords {Variable-order fractional $p(\cdot)$-Laplacian, Double phase problem, Robin boundary condition, variational methods\\
\hspace*{.3cm} {\it 2010 Mathematics Subject Classifications}:
 35R11, 35S15, 47G20, 47J30.}

\begin{abstract}
We consider a fractional double phase Robin problem involving variable order and variable exponents. The
nonlinearity $f$ is a Carath\'{e}odory function satisfying some hypotheses which do not include
the Ambrosetti-Rabinowitz type condition. By using a Variational methods, we investigate
the multiplicity of solutions.
\end{abstract}

\maketitle
\tableofcontents

\section{Introduction}

In the last few decades, problems involving $p(x)$-Laplacian, defined as $(-\Delta)_{p(x)}u := (|\nabla u|^{p(x)-2}\nabla u),\ x \in \RR^N,$ were $p:\RR^N\to[1,\infty)$ is continuous function, have
been studied intensively due to its major real world appearances in several mathematical models, for e.g.,  electrorheological fluid flow, image restorations, etc. (see \cite{ace,chen, Ruzicka, Zhikov}). Various parametric boundary value problems with variable exponents can be found in the book of R\u{a}dulescu-Repov\v{s} \cite{Book-RR} and also one can refer to the book by Diening et al. \cite{diening} for the properties of such operator and associated variable exponent Lebesgue spaces and variable exponent Sobolev spaces.\\
 On the other hand, recently, great attention has been focused on the study of fractional and nonlocal operators of elliptic type, both for pure mathematical research and in view of concrete real-world applications. In most of these applications, a fundamental tool to treat these type of problems
is the so-called fractional order Sobolev spaces. The literature on nonlocal operators and on their applications is very interesting and, up to now, quite large. We also refer to the recent monographs \cite{14,22} for a thorough variational approach of  nonlocal problems.

A bridge between fractional order theories and Sobolev spaces with variable settings is first provided
in \cite{kaufmann}.  In that paper, the authors defined the Fractional Sobolev spaces with variable exponents and introduce the corresponding fractional $p(\cdot)$-Laplacian as
\begin{equation}\label{operator-5}
(-\Delta)_{p(\cdot)}^{s} u(x):=   P.V.\int_{\Om}\frac{\mid
	u(x)-u(y)\mid^{p(x,y)-2}(u(x)-u(y))}{\mid
	x-y\mid^{N+sp(x,y)}}dy, ~~x \in  \Om,
\end{equation} where P.V. denotes Cauchy's principal value, $p:\Om\times\Om\to\RR$ is a continuous function with $1<p(x,y)<\infty$ and $ 0<s<1$, where $\Omega$ is a smooth domain.
The idea of studying such spaces and the associated operator defined in \eqref{operator-5} arises from a natural inquisitiveness to see what results can be recovered when the standard local
$p(x)$-Laplace operator is replaced by the fractional $p(\cdot)$-Laplacian. Continuing with this thought and inspired by the vast applications of variable order derivative (see for e.g. \cite{5,3,1,2,6,4} and references there in), Biswas and Tiwari \cite{rs} introduced  the variable order fractional Sobolev spaces with variable exponent and corresponding variable-order fractional $p(\cdot)$-Laplacian by imposing variable growth on the fractional order $s$, given in \eqref{operator-5}, to study some elliptic problems. In fact, results regarding  fractional $p(\cdot)$-Laplace equations and variable-order fractional $p(\cdot)$-Laplace equations are in progress, for example, we refer to \cite{ab, BR, ky-ho, Ky-Sim} and  \cite{Sabri,res,zfb}, respectively.

In this paper, we are interested in the following problem:

\begin{equation}\label{eq1}
  \begin{cases}
    \mathcal{L}_{p_1,p_2}^s (u)+|u|^{\overline{p}_1(x)-2}u+|u|^{\overline{p}_2(x)-2}u= f(x,u) & \mbox{in}\ \Omega \\
    \mathcal{N}^s_{p_1,p_2}(u)+\beta(x)(|u|^{\overline{p}_1(x)-2}u+|u|^{\overline{p}_2(x)-2}u) =0 & \mbox{in}\ \R^N\setminus\overline{\Omega},
  \end{cases}
\end{equation}
where
$$
\mathcal{L}_{p_1,p_2}^s (u):=(-\Delta)^{s(.)}_{p_1(.)}(u)+(-\Delta)^{s(.)}_{p_2(.)}(u),
$$
$$
\mathcal{N}^{s}_{p_1,p_2}(u):=\mathcal{N}^{s(\cdot)}_{p_1(\cdot)}(u)+\mathcal{N}^{s(\cdot)}_{p_2(\cdot)}(u)
$$
and
\begin{align}\label{operator}
(-\Delta)^{s(.)}_{p_i(.)}u(x)=\displaystyle P.V. \int_{\Om}\frac{|u(x)-u(y)|^{p_i(x,y)}}{|x-y|^{N+s(x,y)p_i(x,y)}}\,d y,\; i=1,2, \quad \text{for }x\in{\Omega},
\end{align}
\begin{align}\label{Neumann boundary}
    \mathcal{N}^{s(\cdot)}_{p_i(\cdot)}u(x)=\int_{\Omega}\frac{|u(x)-u(y)|^{p_i(x,y)-2}(u(x)-u(y))}{|x-y|^{N+s(x,y)p_i(x,y)}}\,dy \quad \text{for }x\in\R^N\setminus\overline{\Omega}.
\end{align}
Here P.V. denotes the Cauchy's principal value, $\Omega\subset\R^N$ is a bounded smooth domain, $s,p_1,p_2$ are  continuous functions such that $\ol{p}_i(x):=p_i(x,x),~i=1,2,~\ol{s}(x):=s(x,x)$ with appropriate assumptions described later. The variable exponent $\beta$ verifies the assumption
\begin{equation}\label{beta}\tag{$\beta$}
  \beta \in L^{\infty}(\R^{N}\setminus \Omega)\quad \text{and}\quad \beta \geq 0\ \text{in}\ \R^{N}\setminus \Omega.
\end{equation}

 The operator, defined in \eqref{eq1}, is called double phase type operator which  has some important applications in biophysics, plasma physics, reaction-diffusion, etc. (see \cite{1.0, 1.1, 1.2}, for e.g.). For more details on applications of such operators in constant exponent set up, that is , $(p,q)$-Laplacian equations, we refer
	to the survey article \cite{21}, see also \cite{Bahrouni-Radulescu-Repovs, Radulescu} for the nonconstant case. This present paper generalizes some results contained in \cite{PRR} and \cite{SRRZ} to the case of nonlocal partial differential equations
with variable exponents.
If $p_1$ and $p_2$ are constants, then \eqref{eq1} becomes the usual nonlocal constant
exponent differential equation discussed in \cite{ambrosioR, Chen-Bao}.  Several results for $(p,q)$-Laplacian problems set in bounded domains and in the whole of $R^N$ can be found in \cite{BBR, Figueiredo, BisciR} and the references therein. But if either $p_1$ or $p_2$ is a non-constant
function, then \eqref{eq1} has a more complicated structure, due to its non-homogeneity and to the presence of several nonlinear terms, only few recent works deal with these problems. For instance, in {\cite{zfb1}, the authors generalize the double phase problem involving a local version of the fractional
	operator with variable exponents, discussed in \cite{tong}, and studied the problem involving variable order fractional $p(\cdot)\&q(\cdot)$-Laplacian but with homogeneous Dirichlet boundary datum, that is, $u=0$ in $\RR^N\setminus\Om.$

Now we consider some notations as follows.
For any set $\mathcal{D}$ and any function $\Phi:\mathcal{D}\rightarrow\mathbb R$, we fix
{\begin{align*}
	\Phi^{-}:=\inf_{\mathcal{D}} \Phi(x)\text{ ~~~and ~~} \Phi^{+}:=\sup_{ \mathcal{D}}\Phi(x).
	\end{align*}}
\noindent We  define the function space
$$C_+(\mathcal{D}):=\{\Phi:\mathcal{D}\to \R \text {~is uniformly continuous~} :~ 1 <\Phi^{-}\leq \Phi^{+}<\infty\}.$$
We consider the following hypotheses on the  variable order $s$ and on the variable exponents $p_1,p_2:$
\begin{itemize}
	\item[{$(H_1)$}] $s:\mathbb R^N\times\mathbb R^N\rightarrow(0,1)$ is a  uniformly continuous and symmetric function, i.e., $s(x,y)=s(y,x)$
	for all $(x,y)\in \RR^N\times \RR^N$
	with $0<s^-\leq  s^+<1$.
	\item[{$(H_2)$}] $p_i\in C_+(\mathbb R^N\times\mathbb R^N)$ are uniformly continuous and symmetric functions, i.e., $p_i(x,y)=p_i(y,x),i=1,2$ for all $(x,y)\in \RR^N\times \RR^N$ 	with  $1<p_1^-\leq p_1^+<p_2^-\leq p_2^+<+\infty$
	such that $s^+p_i^+<N$.
\end{itemize}

 First we study our  problem without assuming the well known Ambrosetti-Rabino\-witz (AR, in short) type condition on the nonlinearity $f$, which is given as
 $$(AR)\qquad\quad {\exists \theta>p_2^+} \text{ s.t.}\;\;tf(x,t)>\theta F(x,t),\;\;\forall |t|>0. $$ {\color{blue} } As known, under $(AR)$, any Palais-Smale sequence of the corresponding energy functional is bounded, which plays an important role of the application of variational methods. In our problem the nonlinearity $f:\Om\times\R\rarrow\R$ is a
 Carath\'eodory function such that $f(x,0)=0$ for a.e. $x\in\Om.$   The further assumptions on $f$ are given below.

\begin{itemize}
\item[{$(f_1)$}] There exists  $a\in L^{\infty}(\Om)$ such that $|f(x,t)|\leq a(x)\left(1+|t|^{r(x)-1}\right),$ for a.e. $x\in\Omega$ and for all $t\in\R,$ where $r\in C_+(\R^N)$ with $p_2^+<r^-\leq r(x)< \frac{N\ol{p}_2(x)}{N-\ol{s}(x)\ol{p}_2(x)}:={{p_2}_{s}}^*(x)$.

\item [{$(f_2)$}] If $F(x,t):=\int_{0}^{t}f(x,s)ds$, then  $\displaystyle\lim_{|t|\to +\infty}\frac{F(x,t)}{|t|^{p_2^+}}=0$ uniformly for $a.e$ $x\in\Om.$

\item [{$(f_3)$}] $\displaystyle\lim_{|t|\to 0}\frac{f(x,t)}{|t|^{p_2^+-2}t}=0$ uniformly for $a.e$ $x\in\Om$.
\item [{$(f_4)$}] Let $\mathcal{F}(x,t)=tf(x,t)-p_2^+ F(x,t).$ Then there exists $b\in L^1(\Om)$ such that
$$
\mathcal{F}(x,t)\leq\mathcal{F}(x,\tau)+b(x)\ \text{for a.e.}\ x\in\Om,\ \text{all}\ 0\leq t\leq \tau\ \text{or all}\ \tau\leq t\leq 0.
$$
\end{itemize}
Consider the following function
      $$	g(x,t)=t|t|^{{\frac{p_2^+}{2}}-2}\log(1+|t|). $$
 One can check that $g$ does not satisfy $(AR)$ but it satisfies $(f_1)$-$(f_4)$. Therefore by dropping $(AR)$ condition, not only we invite complications in the  compactness of Palais-Smale sequence but also we include larger class of nonlinearities. To overcome such aforementioned difficulty, we analyze the Cerami condition (see Definition \ref{cc}), which is more appropriate for the set up of our problem. Finally, we are in a position to state the main results of this article.
\begin{thm}\label{main.result.1}
Let hypotheses $(H_1)$-$(H_2)$, $(\beta)$ and $(f_1)$-$(f_4)$ hold. Then there exists a non-trivial weak solution of \eqref{eq1}.
\end{thm}

 Next, for the odd nonlinearity $f(x,t),$ we state the existence results  of infinitely many solutions using the Fountain theorem and the Dual fountain theorem, respectively.

\begin{thm}\label{fount-sol}
	Let hypotheses $(H_1)$-$(H_2)$, $(\beta)$ and $(f_1)$-$(f_4)$ hold. Also let $f(x,-t)=-f(x,t).$
	Then the problem \eqref{eq1} has a sequence of nontrivial weak solutions with unbounded energy.
\end{thm}

\begin{thm}\label{dual-fount-sol}
	Let hypotheses $(H_1)$-$(H_2)$, $(\beta)$ and $(f_1)$-$(f_4)$ hold. Also let $f(x,-t)=-f(x,t)$.
	Then the problem \eqref{eq1} has a sequence of nontrivial weak solutions with negative critical values converging to zero.
\end{thm}

We prove the next theorem using the symmetric mountain pass theorem.

\begin{thm}\label{sym-infinite-sol}
	Let hypotheses $(H_1)$-$(H_2)$, $(\beta)$ and $(f_1)$-$(f_4)$ hold. Also let $f(x,-t)=-f(x,t)$.
	Then the problem \eqref{eq1} has a  sequence of nontrivial weak solutions with unbounded energy characterized by a minmax argument.
\end{thm}

In the next theorem, we consider the following concave and convex type nonlinearity $f:$
\begin{itemize}
\item[$(f_5)$] For  $\lambda>0$ and $q,r\in C_+(\Om)$ with $1<q^-\leq q^+<p_1^-$ and $p_2^+<r^-$  $$f(x,t)=\lambda |t|^{q(x)-2}t+|t|^{r(x)-2}t$$
\end{itemize}

\begin{thm}\label{clk-infinite-sol}
Let hypotheses $(H_1)$-$(H_2)$, $(\beta)$ and $(f_5)$ hold.
Then for all $\lambda>0,$ the problem \eqref{eq1} has a  sequence of nontrivial weak solutions converging to $0$ with negative energy.
\end{thm}
It is noteworthy to mention that we are the first ( as per the best of our knowledge) to study the above existence results for the problem \eqref{eq1} driven by double phase variable-order fractional $p_1(\cdot)\&p_2(\cdot)$-Laplacian involving Robin boundary condition and non-AR type nonlinearities.
\begin{remark}
Throughout this paper $C$ represents generic positive constant which may vary from line to line.
\end{remark}

\section{Preliminaries results}\label{preli}

\subsection{Variable exponent Lebesgue spaces}

In this section first we recall some basic properties of the  variable exponent Lebesgue spaces, which we will use  to prove our main results.

For $q\in C_+(\Om)$ define the variable exponent Lebesgue space $L^{q(\cdot)}(\Om)$ as
$$
L^{q(\cdot)}(\Om) :={ \Big \{ u : \Om\to\mathbb{R}\  \text{is~measurable}: \int_{\Om} |u(x)|^{q(x)} \;dx<+\infty \Big \}}
$$
which is a separable, reflexive, uniformly convex Banach space {(see \cite{diening,fan})} with respect to the Luxemburg norm
$$
\|{u}\|_{ L^{q(\cdot)}(\Om)}:=\inf\Big\{\eta >0:\int_{\Om}\Big|\frac{u(x)}{\eta}\Big|^{q(x)}\;dx\le1\Big\}.
$$
Define the modular $\rho_{\Om}^{q}:\ L^{q(\cdot)}(\Om)\to\mathbb{R}$ as
$$
\rho_{\Om}^q(u):=\int_{\Om }|u|^{q(x)}\; dx, \  { for~all~} u\in L^{q(\cdot)}(\Om).
$$

\begin{prp} \label{norm-mod}
		{\rm (\cite{fan})}Let $u_n,u\in L^{q(\cdot)}(\Om)\setminus\{0\},$ then the following properties hold:
		\begin{itemize}
			\item[\rm{(i)}]   $\eta=\|u\|_{ L^{q(\cdot)}(\Om)}$ if and only if  $\rho_{\Om}^q(\frac{u}{\eta})=1.$
			\item[\rm{(ii)}] $\rho_{\Om}^q(u)>1$ $(=1;\ <1)$ if and only if  $\|u\|_{ L^{q(\cdot)}(\Om)}>1$ $(=1;\ <1)$,
			respectively.
			\item[\rm{(iii)}] If $\|u\|_{ L^{q(\cdot)}(\Om)}>1$, then $
			\|u\|_{ L^{q(\cdot)}(\Om)}^{p^-}\le \rho_{\Om}^q(u)\le
			\|u\|_{L^{q(\cdot)}(\Om)}^{p^+}$.
			\item[\rm{(iv)}] If $\|u\|_{ L^{q(\cdot)}(\Om)}<1$, then $
			\|u\|_{L^{q(\cdot)}(\Om)}^{p^+}\le \rho_{\Om}^q(u)\le
			\|u\|_{L^{q(\cdot)}(\Om)}^{p^-}$.
			\item[\rm(v)] ${\displaystyle \lim_{n\to  +\infty} }\| u_{n} - u \|_{ L^{q(\cdot)}(\Om)} =0\iff{\displaystyle \lim_{n\to  +\infty}} \rho_{\Om}^q(u_{n} -u)=0.$
		\end{itemize}
\end{prp}

Let $q'$ be the conjugate function of $q,  $ that is,  $1/q(x)+1/q'(x)=1$.

\begin{prp}{\rm (H\"{o}lder inequality)} \label{Holder}{\rm (\cite{fan})}	
	For any $u\in
	L^{q(\cdot)}(\Om)$ and $v\in L^{q'(\cdot)}(\Om)$, we have
	{$$
		\Big|\int_{\Om} uv\,d x\Big|
		\leq
		2\|u\|_{ L^{q(\cdot)}(\Om)}\|{v}\|_{ L^{q'(\cdot)}(\Om)}.
		$$}
\end{prp}

\begin{lemma}{\rm(\cite[Lemma A.1]{sweta})}\label{lemA1}
Let  $\vartheta_1(x)\in L^\infty(\Om)$ such that $\vartheta_1\geq0,\; \vartheta_1\not\equiv 0.$ Let $\vartheta_2:\Om\to \RR$ be a measurable function
such that $\vartheta_1(x)\vartheta_2(x)\geq 1$ a.e. in $\Om.$ Then for every $u\in L^{\vartheta_1(x)\vartheta_2(x)}(\Om),$
$$
\parallel |u|^{\vartheta_1(\cdot)}\parallel_{L^{\vartheta_2(x)}(\Om)}\leq\parallel u\parallel_{L^{\vartheta_1(x)\vartheta_2(x)}(\Om)}^{\vartheta_1^-}+\parallel u\parallel_{L^{\vartheta_1(x)\vartheta_2(x)}(\Om)}^{\vartheta_1^+}.
$$
\end{lemma}

\subsection{Variable order fractional Sobolev spaces with variable exponents}

 Next, we define the fractional Sobolev spaces with variable order and variable exponents (see {\cite{rs}}). Define
\begin{align*}
		W&=	W^{s(\cdot,\cdot),\overline {p}(\cdot),p(\cdot,\cdot)}(\Om)\nonumber\\ &~~~~~~~~:=\Big \{ u\in L^{\overline{p}(\cdot)}(\Om):
		\int_{\Om}\int_{\Om}\frac{| u(x)-u(y)|^{p(x,y)}}{\eta^{p(x,y)}| x-y |^{N+s(x,y)p(x,y)}}dxdy<\infty,
		\text{ for some }\eta>0\Big\}
\end{align*}
endowed with the norm
$$
\|u\|_{W}:=\inf \Big \{\eta>0: \rho_W\left(\frac{u}{\eta}\right) <1 \Big \},
$$
where
$$
\rho_{W}(u):=\int_{\Om}\left|u\right|^{\overline{p}(x)}d x+\int_{\Om}\int_{\Om}\frac{|u(x)-u(y)|^{p(x,y)}}{|x-y|^{N+s(x,y)p(x,y)}}\,d xd y
$$
is a modular on $W.$
Then, $(W, \|\cdot\|_W)$ is a separable reflexive Banach space  {(see \cite{rs,ky-ho})}. On $W$ we also make use of the following norm:
$$
|u|_{W}:=\|u\|_{L^{\overline{p}(\cdot)}(\RR^N)}+[u]_{W},
$$
where the seminorm $[\cdot]_W$ is defined as follows:
$$
[u]_{W}:=\inf \Big\{\eta>0:\int_{\Om}\int_{\Om}\frac{|u(x)-u(y)|^{p(x,y)}}{\eta^{p(x,y)}|x-y|^{N+s(x,y)p(x,y)}}\,d x d y <1 \Big \}.
$$
Note that $\|\cdot\|_{W}$ and $|\cdot|_{W}$ are equivalent norms on $W$  with the relation
\begin{equation}\label{equivalency}
\frac{1}{2}\|u\|_{W}\leq |u|_{W}\leq 2\|u\|_{W}, \ \  \text{~for~all~} u\in W.
\end{equation}

The following embedding result is studied in \cite{rs}. We also refer to \cite{ky-ho} where the authors proved the same result when $s(x,y)=s,$ constant.

\begin{thm}[Sub-critical embedding]\label{Subcritical-embd}
Let $\Omega$ be a smooth bounded  domain in $\mathbb{R}^N$ or $\Omega=\mathbb{R}^N$. Let $s$ and $p$ satisfy $(H_1)$ and $(H_2),$ respectively and $\gamma\in C_+(\ol\Om)$ satisfy $1<\gamma(x)<p_s^*(x)$ for all $x\in\ol\Om$.
In addition, when  $\Omega=\mathbb{R}^N$, $\gamma$ is uniformly continuous and  $\ol {p}(x)<\gamma(x)$ for all $x\in \mathbb{R}^N$ and $\inf_{x\in\R^N}\left(p_s^*(x)-\gamma(x)\right)>0$. Then, it holds that
\begin{equation}\label{subcritical.embedding}
W  \hookrightarrow
L^{\gamma(\cdot)}(\Omega).
\end{equation} Moreover, the embedding is compact.
\end{thm}

{\bf Notations:}
\begin{itemize}
	\item $
	\delta_\Omega^p(u)=\displaystyle\int_{\Om}\int_{\Om}\frac{|u(x)-u(y)|^{p(x,y)}}{|x-y|^{N+s(x,y)p(x,y)}}\,d x d y.
	$
	\item For any measurable set $\mathcal{S},$  $|\mathcal{S}|$ denotes the Lebesgue measure of the set.
\end{itemize}

\section{Functional setting}

Now, we give the variational framework of problem \eqref{eq1}. Let $s,p$ satisfy $(H_1)$, $(H_2),$ respectively. We set
\begin{align*}
    |u|_{{X{_p}}}:=
    [u]_{s(\cdot),p(\cdot),\R^{2N}\setminus (\mathcal{C}\Omega)^2}
    +\|u\|_{L^{\overline{p}(\cdot)}(\Omega)}+\left\|\beta^{\frac{1}{\overline{p}(\cdot)}}u\right\|_{L^{\overline{p}(\cdot)}(\mathcal{C}\Omega)},
\end{align*}
where $\mathcal{C}\Omega=\R^N\setminus\Omega$ and
\begin{align*}
 X^{s(\cdot)}_{p(\cdot)}:=\left\{u\colon\R^N\to \R \text{ measurable } : \  \|u\|_{{X}_{p}}<\infty\right\}.
\end{align*}
By following standard arguments, it can be seen that $X^{s(\cdot)}_{p(\cdot)}$ is reflexive Banach space with respect to the norm $|\cdot|_{X_{p}}$ (see \cite[Proposition 3.1]{BRW}).

Note that the norm $|\cdot|_{X_p}$ is equivalent on $X_{p(\cdot)}^{s(\cdot)}$ to the following norm:
\begin{align}\label{equivalent_norm}
    \begin{split}
     \|u\|_{X_p}&=\inf\left\{\eta\geq 0 \ \bigg| \ \rho_p\left(\frac{u}{\eta}\right)\leq 1\right\}\\
    &=\inf\left\{\eta\geq 0 \ \bigg | \  \int_{\R^{2N}\setminus (\mathcal{C}\Omega)^2}\frac{|u(x)-u(u)|^{p(x,y)}}{\eta^{p(x,y)}p(x,y)|x-y|^{N+s(x,y)p(x,y)}}\,dx\,dy
    + \int_{\Omega}\frac{|u|^{\overline{p}(x)}}{\overline{p}(x) \eta^{\overline{p}(x)}}\,dx\right.\\
    &\qquad \qquad \qquad \quad \left.
    + \int_{\mathcal{C}\Omega}\frac{\beta(x)}{\eta^{\overline{p}(x)}\overline{p}(x)}
    |u|^{\overline{p}(x)}\,dx\leq 1\right\},
    \end{split}
\end{align}
where the modular $\rho_{p}\colon X_{p(\cdot)}^{s(\cdot)}\to \R$ is defined by
\begin{align}\label{mod}
    \rho_{p}\left(u\right)
    &=  \int_{\R^{2N}\setminus (\mathcal{C}\Omega)^2}\frac{|u(x)-u(u)|^{p(x,y)}}{p(x,y)|x-y|^{N+s(x,y)p(x,y)}}\,dx\,dy
    + \int_{\Omega}\frac{|u|^{\overline{p}(x)}}{\overline{p}(x)}\,dx\nonumber\\
    &\quad
    + \int_{\mathcal{C}\Omega}\frac{\beta(x)}{\overline{p}(x)}
    |u|^{\overline{p}(x)}\,dx.
\end{align}

The following lemma will be helpful in later considerations. The proof of this lemma follows using the similar arguments as in \cite{fan}.

\begin{lemma}\label{modular}
    Let $s,p$ and $\beta$ satisfy $(H_1)$, $(H_2)$  and $(\beta),$ respectively, and let $u \in X_{p(\cdot)}^{s(\cdot)}$. Then the following hold:
    \begin{enumerate}
    \item[(i)]
        For $u\neq 0$ we have: $\|u\|_{X_p}=\eta$ if and only if $\rho_p(\frac{u}{\eta})=1$;
    \item[(ii)] If
        $\|u\|_{X_p}<1$ then ${\|u\|_{X_p}^{p{+}}}\leq \rho_p(u)\leq \|u\|_{X_p}^{p^{-}}$;
    \item[(iii)] If
        $\|u\|_{X_p}>1 $ then $ \|u\|_{X_p}^{p^{-}}\leq \rho_p(u)\leq\|u\|_{X_p}^{p^{+}}$.
    \end{enumerate}
\end{lemma}

\begin{lemma}\label{embd}
Let $\Om$ be a smooth bounded domain in $\R^N.$ Let $s$ and $p$ satisfy $(H1)$ and $(H2),$ respectively and $(\beta)$ hold. Then for any $\gamma\in C_+(\ol\Om)$ satisfying $1<\gamma(x)<p_s^*(x)$ for all $x\in\ol\Om,$ there exists a constant $C(s,p,N,\gamma,\Om)>0$ such that
$$
\|u\|_{L^{\gamma(\cdot)}(\Om)}\leq C(s,p,N,\gamma,\Om)\|u\|_{X_p} \text{~~~~for all~~} u\in X,
$$
moreover this embedding is compact.
\end{lemma}

\begin{proof}
It can easily be seen that $\|u\|_W\leq\|u\|_{X_p}.$ Now by applying Theorem \ref{Subcritical-embd}
we get our desired result.
\end{proof}
 In order to deal with fractional $p_1(\cdot)\&p_2(\cdot)$-Laplacian problems, we consider the space
$$
X:=X^{s(\cdot)}_{p_1(\cdot)}\cap X^{s(\cdot)}_{p_2(\cdot)}
$$
endowed with the norm
$$
|u|_X=\|u\|_{{X}_{p_1}}+\|u\|_{{X}_{p_2}}.
$$ Clearly $X$ is reflexive and separable Banach space with respect to the above norm.
It is not difficult to see we can make use of another norm on $X$ equivalent to $|\cdot|_{X}$ given as
$$\|u\|:=\|u\|_{X}= \inf\left\{\eta\geq 0 \ \bigg| \ \rho\left(\frac{u}{\eta}\right)\leq 1\right\},$$ where the modular $\rho:X\to\R$ is defined as $$\rho(u)=\rho_{p_1}(u)+\rho_{p_2}(u)$$ such that $\rho_{p_1},\rho_{p_2}$ are  described  as in \eqref{mod}.

 \begin{lemma}\label{norm-modular}
	Let hypotheses $(H_1)$-$(H_2)$ and $(\beta)$ be satisfied and let $u \in X$. Then the following hold:
	\begin{enumerate}
		\item[(i)]
		For $u\neq 0$ we have: $\|u\|=\eta$ if and only if $\rho(\frac{u}{\eta})=1$;
		\item[(ii)] If
		$\|u\|<1$ then $\|u\|^{p_2^+}\leq\rho(u)\leq \|u\|^{p_1^-}$;
		\item[(iii)] If
		$\|u\|>1 $ then $ \|u\|^{p_1^-}\leq \rho(u)\leq \|u\|^{p_2^+}$.
	\end{enumerate}
\end{lemma}

\begin{lemma}\label{embd-X}
Let $\Om$ be a smooth bounded domain in $\R^N.$ Let $s$ and $p_i$ satisfy $(H1)$ and $(H2),$ respectively for $i=1,2$ and $(\beta)$ hold. Then for any $\gamma\in C_+(\ol\Om)$ satisfying $1<\gamma(x)<{p}_{2_s}^*(x)$ for all $x\in\ol\Om,$
 there exists a constant $C(s,p_i,N,\gamma,\Om)>0$ such that
$$
\|u\|_{L^{\gamma(\cdot)}(\Om)}\leq C(s,p_i,N,\gamma,\Om)\|u\| \text{~~~~for all~~} u\in X,
$$
moreover this embedding is compact.
\end{lemma}

\begin{proof}
The proof directly follows from the definition of $\|u\|$ and Lemma \ref{embd}.
\end{proof}

Throughout this article $X^*$ represents the topological dual of $X$.

\begin{lemma}\label{s+}
Let hypotheses $(H_1)$-$(H_2)$ and $(\beta)$ be satisfied. Then $\rho:X \to\mathbb R$ and $\rho':X\to X^*$ have the following properties:
\begin{itemize}
\item [$(i)$] The function $\rho$ is of class $C^1(X,\mathbb R)$ and $\rho':X\to X^*$ is coercive, that is, $$\frac{\langle\rho'(u), u\rangle}{\|u\|}\to+\infty\text{\;\;as\;}\|u\|\to+\infty.$$
\item[$(ii)$] $\rho'$ is strictly monotone operator.
\item[$(iii)$] $\rho'$  is a mapping of type $(S_+)$, that is, if $u_n\rightharpoonup u$ in $X$ and $\displaystyle\limsup_{n\to+\infty} \langle\rho'(u_n), u_n-u\rangle\leq0$, then $u_n\to u$ strongly in $X$.
\end{itemize}
\end{lemma}

\begin{proof}
 The proof of this result is similar to the proof of \cite[Lemma 4.2]{BR}, just noticing that, the quantities $\R^{2N}\setminus (\mathcal{C}\Omega)^2$ and $\Omega\times\Omega$ play a symmetrical role.
\end{proof}

As proved in \cite[Proposition 3.6]{BRW}, the following integration by parts formula arises naturally for $u\in C^2$ functions:

\begin{align}\label{integrationbypartsformula}
    &\frac{1}{2}\int_{\R^{2N}\setminus(\mathcal{C}\Omega)^2}|u(x)-u(y)|^{p(x,y)-2}\frac{(u(x)-u(y))(v(x)-v(y))}{|x-y|^{N+s(x,y)p(x,y)}}\,dx\,dy\nonumber\\
    &=\int_{\Omega}v(-\Delta)^{s(\cdot)}_{p(\cdot)}u\,dx
    +\int_{\mathcal{C}\Omega}v\mathcal{N}^{s(\cdot)}_{p(\cdot)}\, dx.
    \end{align}

The previous integration by parts formula leads to the following definition:
\begin{definition}
  We say that $u\in X$ is a weak solution to \eqref{eq1} if for any $v\in X$ we have
  \begin{align}\label{weak-formula}
    &\mathcal{H}(u,v)\nonumber\\&=\frac{1}{2}\int_{\R^{2N}\setminus(\mathcal{C}\Omega)^2}\frac{|u(x)-u(y)|^{p_1(x,y)-2}(u(x)-u(y))(v(x)-v(y))}{|x-y|^{N+s(x,y)p_1(x,y)}}dxdy + \int_{\Omega}|u|^{\overline{p}_1(x)-2}uvdx\nonumber\\
    &+\frac{1}{2}\int_{\R^{2N}\setminus(\mathcal{C}\Omega)^2}\frac{|u(x)-u(y)|^{p_2(x,y)-2}(u(x)-u(y))(v(x)-v(y))}{|x-y|^{N+s(x,y)p_2(x,y)}}dxdy +\int_{\Omega}|u|^{\overline{p}_2(x)-2}uvdx\nonumber\\
    &-\int_{\Omega}f(x,u)v\,dx +\int_{\mathcal{C}\Omega} \beta(x)|u|^{\overline{p}_1(x)-2}uv\,dx+\int_{\mathcal{C}\Omega} \beta(x)|u|^{\overline{p}_2(x)-2}uv\,dx\nonumber\\
    &=0.\nonumber\\
  \end{align}
\end{definition}

The problem taken into account in the present paper has a variational structure, namely its solutions can be found as critical points of the associated  energy functional. Hence, our problem can be studied using all the methods which aim to prove the existence of a critical point for a functional.

The energy functional associated with problem \eqref{eq1} is the functional $\mathcal{I}\colon X\to \R$ given by

\begin{align*}
  \mathcal{I}(u) & = \frac{1}{2}\int_{\R^{2N}\setminus(\mathcal{C}\Omega)^2}\frac{|u(x)-u(y)|^{p_1(x,y)}}{p_1(x,y)|x-y|^{N+s(x,y)p_1(x,y)}}dxdy+
  \int_{\Omega}\frac{1}{\overline{p}_1(x)}|u|^{\overline{p}_1(x)}dx\\
  &+\frac{1}{2}\int_{\R^{2N}\setminus(\mathcal{C}\Omega)^2}\frac{|u(x)-u(y)|^{p_2(x,y)}}{p_2(x,y)|x-y|^{N+s(x,y)p_2(x,y)}}dxdy+  \int_{\Omega}\frac{1}{\overline{p}_2(x)}|u|^{\overline{p}_2(x)}dx\\
  &+\int_{\mathcal{C}\Omega} \frac{\beta(x)|u|^{\overline{p}_1(x)}}{\overline{p}_1(x)}\,dx+\int_{\mathcal{C}\Omega} \frac{\beta(x)|u|^{\overline{p}_2(x)}}{\overline{p}_2(x)}\,dx-\int_{\Omega}F(x,u)dx.
\end{align*}

A direct computation from \cite[Proposition 3.8]{BRW} shows that the functional $\mathcal{I}$ is well defined on $X$ and $\mathcal{I}\in C^1(X,\R)$ with
$$
\langle \mathcal{I}^{'}(u),v\rangle=\mathcal{H}(u,v)\quad \text{for any}\quad v\in X.
$$
Thus the weak solutions of \eqref{eq1} are precisely the critical points of $\mathcal{I}$.

\section{Proof of Theorem \ref{main.result.1}}

\subsection{Abstract results}

\begin{definition}\label{cc}
  Let $E$ be a Banach space and $E^*$ be its topological dual. Suppose that $\Phi\in C^1(E)$. We say that $\Phi$ satisfies the Cerami condition at the level $c\in\R$ $($the $(C)_c$-condition for short$)$ if the following is true:
  $$
  ``\text{every sequence}\ (u_n)_{n\in\mathbb{N}}\subseteq E\ \text{such that}\ \Phi(u_n)\to c\ \text{and}
  $$
  $$
  (1+\|u_n\|_E)\Phi^{'}(u_n)\to 0\ \text{in}\ E^*\ \text{as}\ n\to+\infty
  $$
  $$
  \text{admits a strongly convergent subsequence}".
  $$
  If this condition holds at every level $c \in\R$, then we say that $\Phi$ satisfies the Cerami condition (the $C$-condition for short).
\end{definition}

The $(C)_c$-condition is weaker than the $(PS)_c$-condition. However, it was shown in \cite{KP} that from $(C)_c$-condition it can obtain a deformation lemma, which is fundamental in order to get some minimax theorems. Thus we have

\begin{thm}\label{MPT}
  If there exist $e\in E$ and $r>0$ such that
  $$
  \|e\|>r,\quad \max(\Phi(0),\Phi(e))\leq \inf_{\|x\|=r}\Phi(x),
  $$
  and
  $\Phi\colon E\to \R$ satisfies the $(C)_c$-condition with
  $$
  c=\inf_{\gamma\in \Gamma}\max_{t\in(0,1)} \Phi(\gamma(t)),
  $$
  where
  $$
  \Gamma=\{\gamma\in C((0,1),E)\colon \gamma(0)=0,\ \gamma(1)=e\}.
  $$
  Then $c\geq \inf_{\|x\|=r}\Phi(x)$ and $c$ is a critical value of $\Phi$.
\end{thm}

\subsection{Geometric condition}

\begin{lemma}\label{geo}
 Let $(H_1)$-$(H_2)$, $(\beta)$ and $(f_1)$-$(f_4)$ hold. Then
\begin{enumerate}
  \item [$(i)$] there exist $\alpha>0$ and $R>0$ such that
    \begin{align*}
    \mathcal{I}(u)\geq \beta>0 \quad \text{for any } u\in X \text{ with }\ \|u\|=\alpha.
    \end{align*}
  \item [$(ii)$] there exists $\varphi \in X$ such that $ I(\varphi)<0$.
\end{enumerate}
\end{lemma}
\begin{proof}
  $(i)$ For any  $\epsilon > 0$, by the assumptions $(f_1)$-$(f_3)$, we have
  \begin{equation}\label{g1}
    F(x,t)\leq \epsilon|t|^{p_2^+}+C(\epsilon)|t|^{r(x)},\quad \text{for a.e.}\ x\in\Omega\ \text{and all}\ t\in \R.
  \end{equation}

\noindent Hence, using Theorem \ref{Subcritical-embd}, Lemma \ref{norm-modular}, Lemma \ref{lemA1} and Lemma \ref{embd-X} for any $u\in X$ with $\|u\|<1$ $($ i.e. $\|u\|_{X_{p_i}}<1, i=1,2 ),$ we obtain
    \begin{align*}
   \mathcal {I}(u)
    &= \frac{1}{2}\int_{\R^{2N}\setminus(\mathcal{C}\Omega)^2}\frac{|u(x)-u(y)|^{p_1(x,y)}}{p_1(x,y)|x-y|^{N+s(x,y)p_1(x,y)}}dxdy+
    \int_{\Omega}\frac{1}{\overline{p}_1(x)}|u|^{\overline{p}_1(x)}dx\\
    &+\frac{1}{2}\int_{\R^{2N}\setminus(\mathcal{C}\Omega)^2}\frac{|u(x)-u(y)|^{p_2(x,y)}}{p_2(x,y)|x-y|^{N+s(x,y)p_2(x,y)}}dxdy+  \int_{\Omega}\frac{1}{\overline{p}_2(x)}|u|^{\overline{p}_2(x)}dx\\
    &+\int_{\mathcal{C}\Omega} \frac{\beta(x)|u|^{\overline{p}_1(x)}}{\overline{p}_1(x)}\,dx+\int_{\mathcal{C}\Omega} \frac{\beta(x)|u|^{\overline{p}_2(x)}}{\overline{p}_2(x)}\,dx\\
    &\quad - \int_{\Omega}F(x,u)\,dx\\
    &\geq \frac{1}{2}\rho(u)- { \e}\int_{\Omega}|u|^{p_2^+}\,dx- C(\epsilon)\int_{\Omega}|u|^{r(x)}\,dx\\
    &\geq \frac{1}{2}\|u\|^{p_2^{+}}
    -\e\|u\|^{p_2^{+}}_{L^{p_2^+}(\Omega)}- C(\epsilon)\left\{\|u\|^{r^{-}}_{L^{r(\cdot)}(\Omega)}+\|u\|^{r^{+}}_{L^{r(\cdot)}(\Omega)}\right\}\\
    &\geq\frac{1}{2}\|u\|^{p_2^{+}}-\e C\|u\|^{p_2^{+}}- C^{'}(\epsilon)\|u\|^{r^-}\\
    &=\left(\frac{1}{2}-\e C\right)\|u\|^{p_2^{+}}- C^{'}(\epsilon)\|u\|^{r^-},
    \end{align*}
where $C^{'}(\epsilon)>0$ is a constant. Consider
$$
0<\e<\frac{1}{4C}.
$$
Since $p_2^+<r^-$, we can choose $\alpha\in(0,1)$ sufficiently small such that for all $u\in X$ with $\|u\|=\alpha$
\begin{align*}
    \mathcal{I}(u)\geq \alpha^{p_2^{+}}\left(\frac{1}{2}-{\epsilon C}\right)-C^{'}(\epsilon)\alpha^{r^{-}}=R>0.
    \end{align*} The proof of $(i)$ is complete.

  \medskip

  $(ii)$ It follows from $(f_1)$ and $(f_2)$ that for any positive constant $M$, there exists a corresponding positive
constant $C_M$ such that
\begin{align}\label{g2}
F(x,t)\geq {M}|t|^{p_2^+}-C_M.
\end{align}
Let $e\in X, ~e>0$ with $\|e\|=1$ and $\int_{\Om}|e|^{p_2^+}dx>0 $ and $t>1$. Then, using Lemma \ref{norm-modular} and \eqref{g2}, we get
  \begin{align*}
    I(t e)
    &=\int_{\R^{2N}\setminus(\mathcal{C}\Omega)^2}t^{p_1(x,y)} \frac{|e(x)-e(y)|^{p_1(x,y)}}{2p_1(x,y)|x-y|^{N+sp_1(x,y)}}\,dx\,dy
    +\int_{\mathcal{C}\Omega}t^{\overline{p}_1(x)} \frac{\beta(x)|e|^{\overline{p}_1(x)}}{\overline{p}_1(x)}\,dx\\
    &\quad +\int_{\Omega}t^{\overline{p}_1(x)}\frac{|e|^{\overline{p}_1(x)}}{\overline{p}_1(x)}\,dx\\
    &+\int_{\R^{2N}\setminus(\mathcal{C}\Omega)^2}t^{p_2(x,y)} \frac{|e(x)-e(y)|^{p_2(x,y)}}{2p_2(x,y)|x-y|^{N+sp_2(x,y)}}\,dx\,dy
    +\int_{\mathcal{C}\Omega}t^{\overline{p}_2(x)} \frac{\beta(x)|e|^{\overline{p}_2(x)}}{\overline{p}_2(x)}\,dx\\
    &\quad +\int_{\Omega}t^{\overline{p}_2(x)}\frac{|e|^{\overline{p}_2(x)}}{\overline{p}_2(x)}\,dx
    -  \int_{\Omega}F(x,te)\,dx\\
    &\leq t^{p_2^+} \rho(e)
 -t^{p_2^+} M\int_{\Omega}|e|^{p_2^+}\,dx+ |\Omega| C_M\\
 &=t^{p_2^+} \left[1- M\int_{\Omega}|e|^{p_2^+}\,dx\right]+|\Omega| C_M
    \end{align*}
    We choose $M$ sufficiently large  so that
    $$\lim_{t\to+\infty}\mathcal{I}(t e)=-\infty.$$ Hence, there exists some $t_0>0$ such that $\mathcal{I}(\varphi)<0,$ where $\varphi=t_0 e.$ Thus the proof of $(ii)$ is complete.
\end{proof}

\subsection{Cerami condition}

\begin{prp}\label{bounded}
  If hypotheses $(H_1)$-$(H_2)$, $(\beta)$ and $(f_1)$-$(f_4)$ hold, then  the functional $\mathcal{I}$ satisfies the $(C)_c$-condition for any $c\in\R.$
\end{prp}

\begin{proof}
  In this proof the value of the constant $C$ changes from line to line. We consider a sequence $(u_n)_{n\geq1}\subset X$ such that
  \begin{equation}\label{C1}
    |\mathcal{I}(u_n)|\leq C\quad\text{for some}\quad C>0\quad\text{and for all}\ n\geq1,
  \end{equation}
  \begin{equation}\label{C2}
    (1+\|u_n\|)\mathcal{I}^{'}(u_n)\to 0\quad\text{in}\quad X^{*}\quad\text{as}\ n\to+\infty.
  \end{equation}
  From \eqref{C2}, we have
  \begin{equation}\label{C3}
    |\mathcal{H}(u_n,v)|\leq \frac{\epsilon_n\|v\|}{1+\|u_n\|},
  \end{equation}
  for all $v\in X$ with $\epsilon_n\to 0$.

  In \eqref{C3}, we choose $v=u_n\in X$ and obtain for all $n\in\mathbb N$
  \begin{align}\label{C4}
   &-\frac{1}{2}\delta^{p_1}_{\R^{2N}\setminus(\mathcal{C}\Omega)^2}(u_n)-\frac{1}{2}\delta^{p_2}_{\R^{2N}\setminus(\mathcal{C}\Omega)^2}(u_n)
   -\rho_\Omega^{p_1}(u_n) -\rho_\Omega^{p_2}(u_n)\nonumber\\
   &-\int_{C \Omega}\beta(x)|u_n|^{\overline{p}_1(x)}dx-\int_{C \Omega}\beta(x)|u_n|^{\overline{p}_2(x)}dx+\int_{\Omega}f(x,u_n(x))u_n(x)dx\nonumber\\
   &\leq \epsilon_n.
  \end{align}
  Also, by \eqref{C1} we have for all $n\in\mathbb N$,
  \begin{align}\label{C5}
    &\frac{1}{2p_1^+}\delta^{p_1}_{\R^{2N}\setminus(\mathcal{C}\Omega)^2}(u_n)+ \frac{1}{2p_2^+}\delta^{p_2}_{\R^{2N}\setminus(\mathcal{C}\Omega)^2}(u_n)+\frac{1}{p_1^+}\rho_\Omega^{p_1}(u_n)
    +\frac{1}{p_2^+}\rho_\Omega^{p_2}(u_n)\nonumber\\
    &+\frac{1}{p_1^+}\int_{C \Omega}\beta(x)|u_n|^{\overline{p}_1(x)}dx+\frac{1}{p_2^+}\int_{C \Omega}\beta(x)|u_n|^{\overline{p}_2(x)}dx\nonumber\\
    &-\int_{\Omega}F(x,u_n(x))dx\leq  C.
  \end{align}
  Adding relations \eqref{C4} and \eqref{C5}, we obtain
  \begin{equation}\label{C6}
     \int_{\Omega}\mathcal{F}(x,u_n(x))dx\leq C\quad\text{for some}\quad C>0\quad \text{and all}\ n\in\mathbb N.
  \end{equation}
  {\bf Claim:} The sequence $(u_n)_{n\geq1}\subset X$ is bounded.

  We argue by contradiction. Suppose that the claim is not true. We may assume that
  \begin{equation}\label{C7}
    \|u_n\|\to +\infty\quad\text{as}\quad n\to+\infty.
  \end{equation}
  We set $w_n:=\frac{u_n}{\|u_n\|}$ for all $n\in\mathbb N$. Then $\|w_n\|=1$, for all $n\in\mathbb N$. Using reflexivity of $X$ and Lemma \ref{embd-X}, up to a subsequence, still denoted by $(w_n)_{n\geq1}$, as $n\to+\infty,$ we get
  \begin{equation}\label{C8}
    w_n\rightharpoonup w\ \text{~weakly in}\ X\quad\text{and}\quad w_n\to w\ \text{~ strongly in}\ L^{\gamma(\cdot)}(\Omega),\; 1<\gamma(x)<{p_2}_s^*(x).
  \end{equation}

  We claim that $w =0$. Indeed, if not then the set $\widehat{\Omega}:=\{x\in\Om:w(x)\not=0\}$ has positive Lebesgue measure, i.e., $|\widehat{\Omega}|>0.$ Hence, $|u_n(x)|\to +\infty$ for a.e.  $x\in\widehat{\Omega}$ as $n\to+\infty$. On account of hypothesis $(f_2)$, for a.e. $x\in\widehat{\Omega} $ we have
\begin{align}\label{C0}
 \frac{F(x,u_n(x))}{\|u_n\|^{p_2^+}}=\frac{F(x,u_n(x))}{|u_n(x)|^{p_2^+}}|w_n(x)|^{p_2^+}\to+\infty\ \text{~as~} n\to+\infty.
\end{align}

  Then by Fatou's lemma, we obtain
  \begin{equation}\label{C9}
    \int_{\widehat{\Omega}}\frac{F(x,u_n(x))}{\|u_n\|^{p_2^+}}dx\to+\infty \text{~~as~} {n\to+\infty}.
  \end{equation}
 Hypotheses $(f_1)$-$(f_2)$ imply there exists $K> 0$ such that
  \begin{equation}\label{C10}
    \frac{F(x,t)}{|t|^{p_2^+}}\geq 1\quad \text{for a.e.}\ x\in\Omega,\ \text{all}\ |t|>K.
  \end{equation}
  By $(f_1),$ there exists a positive constant
 $\widehat{C}>0$ such that
 \begin{align}\label{C10'}
 |F(x,t)|\leq \widehat{C}, \text{ for all~} (x,t)\in \ol{\Om}\times[-K,K].
 \end{align}  Now from \eqref{C10} and \eqref{C10'}, we get
 \begin{align}\label{C10"}
 F(x,t)>C_0 \text{ ~for all~} (x,t)\in \ol{\Om}\times\R,
 \end{align} where $C_0\in\R$ is a constant. The above relation implies
 \begin{align*}
 \frac{F(x,u_n(x))-C_0}{\|u_n\|^{p_2^+}}\geq0 \text{~~ for all~} x\in \ol{\Om}, \text{ for all~} n\in\mathbb{N}.
 \end{align*}
  that is,
 \begin{align}\label{C10.0}
 \frac{F(x,u_n(x))}{|u_n(x)|^{p_2^+}}|w_n(x)|^{p_2^+}-\frac{C_0}{\|u_n\|^{p_2^+}}\geq0 \text{~~ for all~} x\in \ol{\Om}, \text{ for all~} n\in\mathbb{N}.
 \end{align}

  By \eqref{C1}, \eqref{C7}, \eqref{C9}, \eqref{C10.0}  and using the fact $\|w_n\|=1,$  Lemma \ref{modular} and Fatou's lemma, we have
  \begin{align}\label{C12}
 +\infty&= \left[\int_{\widehat{\Om}}\liminf_{n\to+\infty}\frac{F(x,u_n(x))|w_n(x)|^{p_2^+}}{|u_n(x)|^{p_2^+}}dx-\int_{\widehat{\Om}}\limsup_{n\to+\infty}\frac{C_0}{\|u_n\|^{p_2^+}}dx\right]\nonumber\\
  &=\int_{\widehat{\Om}}\liminf_{n\to+\infty}\left[\frac{F(x,u_n(x))|w_n(x)|^{p_2^+}}{|u_n(x)|^{p_2^+}}-\frac{C_0}{\|u_n\|^{p_2^+}}\right]dx\nonumber\\
  &\leq\liminf_{n\to+\infty}\int_{\widehat{\Om}}\left[\frac{F(x,u_n(x))|w_n(x)|^{p_2^+}}{|u_n(x)|^{p_2^+}}-\frac{C_0}{\|u_n\|^{p_2^+}}\right]dx\nonumber\\
  &\leq\liminf_{n\to+\infty}\int_{{\Om}}\left[\frac{F(x,u_n(x))|w_n(x)|^{p_2^+}}{|u_n(x)|^{p_2^+}}-\frac{C_0}{\|u_n\|^{p_2^+}}\right]dx\nonumber\\
  &=\left[\liminf_{n\to+\infty}\int_{{\Om}}\frac{F(x,u_n(x))|w_n(x)|^{p_2^+}}{|u_n(x)|^{p_2^+}}dx-\limsup_{n\to+\infty}\int_{{\Om}}\frac{C_0}{\|u_n\|^{p_2^+}}dx\right]\nonumber\\
   &=\liminf_{n\to+\infty}\int_{{\Om}}\frac{F(x,u_n(x))}{\|u_n\|^{p_2^+}}dx\nonumber\\
  &=\liminf_{n\to+\infty}\bigg[\frac{1}{2}\int_{\R^{2N}\setminus(\mathcal{C}\Omega)^2}
    \frac{1}{\|u_n\|^{p_2^+-p_1(x,y)}}\frac{|w_n(x)-w_n(y)|^{p_1(x,y)}}{p_1(x,y)|x-y|^{N+s(x,y)p_1(x,y)}}dxdy\nonumber\\
    &\quad+\frac{1}{2}\int_{\R^{2N}\setminus(\mathcal{C}\Omega)^2}
    \frac{1}{\|u_n\|^{p_2^+-p_2(x,y)}}\frac{|w_n(x)-w_n(y)|^{p_2(x,y)}}{p_2(x,y)|x-y|^{N+s(x,y)p_2(x,y)}}dxdy\nonumber\\
    &\quad+ \int_{\Omega}\frac{1}{\|u_n\|^{p_2^+ -\overline{p}_1(x)}}\frac{|w_n|^{\overline{p}_1(x)}}{\ol{p}_1(x)}dx
    + \int_{\Omega}\frac{1}{\|u_n\|^{p_2 ^+ -\overline{p}_2(x)}}\frac{|w_n|^{\overline{p}_2(x)}}{\ol{p}_2(x)}dx\nonumber\\
    &\quad+\int_{\mathcal{C}\Omega} \frac{\beta(x)|w_n|^{\overline{p}_1(x)}}{\|u_n\|^{p_2^+ -\overline{p}_1(x)}\ol{p}_1(x)}\,dx
    +\int_{\mathcal{C}\Omega} \frac{\beta(x)|w_n|^{\overline{p}_2(x)}}{\|u_n\|^{p^+ -\overline{p}_2(x)}\ol{p}_2(x)}dx-\frac{\mathcal{I}(u_n)}{\|u_n\|^{p^+}}\bigg]\nonumber\\
  &\leq\liminf_{n\to+\infty} \rho(w_n)=1.
  \end{align}
 Thus we arrive at a contradiction.
Hence, $w = 0$. Let $\mu\geq1$ and set $\kappa
 :=(2\mu)^{\frac{1}{p_2^-}}\geq1$ for all $n\in\mathbb N$. Evidently, from \eqref{C8} we have $$ w_n\to 0\quad\text{~ strongly in}\quad L^{\gamma(\cdot)}(\Omega),~ 1<\gamma(x)<{p_2}_s^*(x)$$ which combining with $(f_1)$-$(f_3)$ and Lebesgue dominated convergence theorem yields that
  \begin{align}\label{C13}
  \int_{\Omega}F(x,\kappa w_n)\,dx\to0\ \text{as}\ n\to+\infty.
  \end{align}
We can find $t_n\in[0,1]$ such that
\begin{equation}\label{C15}
  \mathcal{I}(t_n u_n)=\max_{0\leq t\leq 1}\mathcal{I}(tu_n).
\end{equation}
Because of \eqref{C7}, for sufficiently large $n\in\mathbb N,$ we have
\begin{equation}\label{C16}
  0<\frac{(2\mu)^{\frac{1}{p_2^-}}}{\|u_n\|}\leq1.
\end{equation}
Using \eqref{C13}, \eqref{C15} and  \eqref{C16} and recalling that $ \|w_n\|=1,$ for sufficiently large $n\in\mathbb N,$ it follows that
\begin{align*}
   \mathcal{I}(t_n u_n) \geq \mathcal{I}\left(\kappa \frac{u_n}{\|u_n\|}\right)&= \mathcal{I}(\kappa w_n)\\
&\geq (\kappa)^{p_2^-} \frac{1}{2}\rho(w_n)-\int_{\Omega}F(x,\kappa w_n)dx\\
&=2\mu. \frac{1}{2}\rho(w_n)-\int_{\Omega}F(x,\kappa w_n)dx\\
   &= \mu + o_n(1).
\end{align*}
Since $\mu > 0$ is arbitrary, we have
\begin{equation}\label{C17}
  \mathcal{I}(t_n u_n)\to+\infty\quad\text{as}\quad n\to+\infty.
\end{equation}
From the assumption$f(x,0)=0$ and \eqref{C1} we know that
\begin{equation}\label{C18}
  \mathcal{I}(0)=0\ \text{and}\ \mathcal{I}(u_n)\leq C\ \text{for all}\ n\in\mathbb N.
\end{equation}
By \eqref{C17} and \eqref{C18}, we can infer that, for $n\in\mathbb N$ large,
\begin{equation}\label{C19}
  t_n\in(0,1).
\end{equation}
From \eqref{C15} and \eqref{C19}, we can see that for all $n\in\mathbb N$ sufficiently large,
\begin{align}\label{C20}
  0=t_n\frac{d}{dt}\mathcal{I}(t u_n)|_{t=t_n}=\langle \mathcal{I}^{'}(t_n u_n),t_n u_n\rangle,
\end{align}
so,
\begin{align}\label{C21}
  &\frac{1}{2}\delta^{p_1}_{\R^{2N}\setminus(\mathcal{C}\Omega)^2}(t_n u_n) + \rho^{p_1}_\Omega(t_n u_n)
   +\int_{\mathcal{C}\Omega} \beta(x)|t_n u_n|^{\overline{p}_1(x)}\,dx\nonumber\\
   &+\frac{1}{2}\delta^{p_2}_{\R^{2N}\setminus(\mathcal{C}\Omega)^2}(t_n u_n) + \rho^{p_2}_\Omega(t_n u_n)
   +\int_{\mathcal{C}\Omega} \beta(x)|t_n u_n|^{\overline{p}_2(x)}\,dx
    -\int_{\Omega}f(x,t_n u_n)t_n u_n\,dx =0.
\end{align}
From hypothesis $(f_4)$, we obtain for all $n\in\mathbb N,$
$$
\mathcal{F}(x,t_n u_n)\leq\mathcal{F}(x, u_n)+b(x)\ \text{for a.e}\ x\in \Omega,
$$
that is,
\begin{equation}\label{C22}
  f(x,t_n u_n)(t_n u_n)\leq\mathcal{F}(x, u_n)+b(x)+p_2^+ F(x,t_n u_n)\ \text{for a.e}\ x\in \Omega.
\end{equation}

Combining \eqref{C21} and  \eqref{C22}, we deduce
{\begin{align*}&\frac{1}{2}\delta^{p_1}_{\R^{2N}\setminus(\mathcal{C}\Omega)^2}(t_n u_n) + \rho^{p_1}_\Omega(t_n u_n)
	+\int_{\mathcal{C}\Omega} \beta(x)|t_n u_n|^{\overline{p}_1(x)}\,dx\\
   &+\frac{1}{2}\delta^{p_2}_{\R^{2N}\setminus(\mathcal{C}\Omega)^2}(t_n u_n) + \rho^{p_2}_\Omega(t_n u_n)+
   \int_{\mathcal{C}\Omega} \beta(x)|t_n u_n|^{\overline{p}_2(x)}\,dx-p_2^+\int_\Omega F(x,t_n u_n)dx\\
   & \leq \int_\Omega \mathcal{F}(x,u_n)dx+\|b\|_{L^1(\Om)}\ \text{for all}\ n\in\mathbb N,
\end{align*}}
and hence by \eqref{C6}, we get
\begin{equation}\label{C23}
  p_2^+\mathcal{I}(t_n u_n)\leq C\ \text{for all}\ n\in\mathbb N.
\end{equation}
We compare \eqref{C17} and \eqref{C23} and arrive at a contradiction. Thus the claim follows.

On account of this claim, we may assume that
\begin{equation}\label{C24}
    u_n\rightharpoonup u\ \text{~ weakly in}\ X\quad\text{and}\quad u_n\to u\ \text{~ strongly in}\ L^{\gamma(\cdot)}(\Omega),\ 1<\gamma(x)<{p_2}_s^*(x).
  \end{equation}
  We show in what follows that
  $$
  u_n\to u\ \text{in}\ X.
  $$
  Using \eqref{C24}, we have
\begin{align}\label{C25}
  o_n(1)=\langle\mathcal{I}^{'}(u_n),u_n-u\rangle\geq\frac 12\langle\rho'(u_n), u_n-u\rangle-\int_\Om f(x,u_n)(u_n-u)dx.
\end{align}
  Now by $(f_1),$ H\"older inequality, \eqref{C24},  boundedness of $(u_n)_n$ and Lemma \ref{lemA1}, we obtain
  \begin{align}\label{C26}
  &\left|\int_{\Om} f(x,u_n)(u_n-u)dx\right|\nonumber\\&\leq\|a\|_{L^\infty(\Om)}\left[\int_\Om |u_n-u|.1\,dx+\int_\Om|v_n|^{r(x)-1}|v_n-v_0|dx\right]\nonumber\\
  &\leq\|a\|_{L^\infty(\Om)}\left[ \|u_n-u\|_{L^{r(\cdot)}(\Om)} \left(1+|\Om|\right)^{r'^+}+ \|v_n-v_0\|_{L^{r(\cdot)}(\Om)}
  \||u_n|^{r(\cdot)-1}\|_{L^{r'(\cdot)}(\Om)}\right]\nonumber\\
  &\leq C \left[ \|u_n-u\|_{L^{r(\cdot)}(\Om)}+ \|u_n-u\|_{L^{r(\cdot)}(\Om)}
  \left(\|u_n\|_{L^{r(\cdot)}(\Om)}^{r^+-1}+\|u_n\|_{L^{r(\cdot)}(\Om)}^{r^--1}\right)\right]\nonumber\\
  &\quad\to 0\text{\;\; as\;} n\to+\infty.\nonumber\\
  \end{align}
  Hence, combining \eqref{C25} and \eqref{C26} and using the $(S_+)$ property of $\rho'$ (see Lemma \ref{s+}),
  	we have $u_n\to u$ strongly in $X$ as $n\to+\infty.$, which
shows that the $(C)_c$-condition is satisfied. The proof is now complete.
\end{proof}

\section{Proof of Theorem \ref{fount-sol}}

To prove the Theorem \ref{fount-sol} we need the Fountain theorem of Bartsch \cite[Theorem 2.5]{bartsch}; (see also \cite[Theorem 3.6]{willem}).
We recall next lemma from \cite{fabian}.

\begin{lemma}\label{ftlem}
	Let $E$ be a reflexive and separable Banach space. Then there are $\{e_n\}\subset E$ and $\{g_n^*\}\subset E^*$ such that{$$E=\overline{span\{e_n:n=1,2,3..\}}, ~~E^*=\overline{span\{g_n^*:n=1,2,3..\}},$$} and { \begin{equation*}
		\langle g_i^*,e_j\rangle=
		\left\{ \begin{array}{rl}
		& 1  \text{~~~~if~~~} i=j\\
		& 0 \text{~~~~if~~~} i\not=j.
		\end{array}
		\right.
		\end{equation*}}
\end{lemma}

Let us denote
\begin{align}\label{not}
	E_n=span\{e_n\},~~~ X_k=\bigoplus_{n=1}^k E_n \text{ ~~~and } Y_k=\overline{\bigoplus_{n=k}^\infty E_n}.
\end{align}
Now we recall the following Fountain theorem from  \cite{alves}:

\begin{thm}[Fountain theorem]\label{ft}
	Assume that $\Phi\in C^1(E,\RR)$ satisfies the Cerami condition $(C)_c$ for all $c\in\mathbb R$ and $\Phi(-u)=\Phi(u).$ If for each sufficiently large $k\in\mathbb N,$ there exists $\varrho_k>\delta_k>0$ such that
	\begin{itemize}
		\item[${(\mathcal A_1)}$] $b_k:=\inf\{\Phi(u):u\in Y_k,~\|u\|_E=\delta_k\}\to+\infty,$ as $k\to+\infty,$
		\item[$\rm(\mathcal A_2)$] $a_k:=\max\{\Phi(u):u\in X_k,~\|u\|_E=\varrho_k\}\leq0.$
	\end{itemize}
	Then $\Phi$ has a sequence of critical points $(u_k)_k$ such that $\Phi(u_k)\to+\infty.$
\end{thm}

\begin{proof}[Proof of Theorem \ref{fount-sol}:]  Define $X_k$ and $Y_k$ as in \eqref{not} for the reflexive, separable Banach space $X$. Now  $\mathcal{I}$ is even and satisfies Cerami condition $(C)_c$  for all $c\in \RR$ (see Lemma \ref{bounded}). So now we will  show that the conditions $(\mathcal A_1)$-$(\mathcal A_2)$ hold for our problem.
\begin{itemize}
	\item [$(\mathcal A_1):$]
	For large $k\in\mathbb N,$   set
	\begin{align}\label{al}
	\alpha_k=\sup_{u\in Y_k,~\|u\|=1}\|u\|_{L^{\gamma(\cdot)}(\Om)},\end{align} where $\gamma\in C_+(\ol\Om)$ such that for all $x\in\ol\Om,$ $1< \gamma(x)<p_s^*(x).$  So,
	\begin{align}\label{ft1}
	\lim_{k\to+\infty}\alpha_k= 0.
	\end{align}
Supposing to the contrary, there exist $\epsilon'>0, k_0\geq0$ and a sequence $(u_k)_k$ in $Y_k$ such that
$$\|u_k\|=1\text{ and}~~\|u_k\|_{L^{\gamma(\cdot)}(\Om)}\geq \epsilon' $$
	for all $k\geq k_0.$
	Since $(u_k)_k$ is bounded in $X,$ there exists $u_0\in E$ such that up to a subsequence, still denoted by $(u_k)_k,$ we have $u_k\rightharpoonup u_0$ weakly in $E$ as $k\to+\infty$ and
	$$\langle g_j^*,u_0\rangle=\lim_{k\to+\infty}\langle g_j^*,u_k\rangle=0$$  for $j=1,2,3,\cdots.$ Thus we have $u=0.$ In addition, using Theorem \ref{embd-X} we obtain
	$$\epsilon'\leq\lim_{k\to+\infty}\|u_k\|_{L^{\gamma(\cdot)}(\Om)}=\|u_0\|_{L^{\gamma(\cdot)}(\Om)}=0,$$  a contradiction. Hence, \eqref{ft1} holds true.  Let $u\in Y_k$ with  $\|u\|>1.$  Note that  \eqref{ft1} implies $\alpha_k<1$ for large $k\in\mathbb N.$ Thus using  Lemma \ref{lemA1}, Lemma \ref{norm-modular} and  \eqref{al} and \eqref{g1} with $\e=1$ for  $k\in\mathbb N$ large enough,   we get
\begin{align}\label{ft4}
 \mathcal {I}(u)
	&\geq \frac{1}{2}\rho(u)-\left[\int_{\Omega}|u|^{p_2^+}\,dx- C(1)\int_{\Omega}|u|^{r(x)}\,dx\right] \nonumber\\
		&\geq \frac{1}{2}\|u\|^{p_1^-}
		-\|u\|^{p_2^{+}}_{L^{p_2^+}(\Omega)}- C(1)\left\{\|u\|^{r^{-}}_{L^{r(\cdot)}(\Omega)}+\|u\|^{r^{+}}_{L^{r(\cdot)}(\Omega)}\right\} \nonumber\\
		&\geq\frac{1}{2}\|u\|^{p_1^-}
		- \alpha_k^{p_2^+} C_1\|u\|^{p_2^{+}}- C_2\{\alpha_k^{r^-}\|u\|^{r^{-}}+\alpha_k^{r^+}\|u\|^{r^{+}}\}\nonumber\\
		&\geq\frac{1}{2}\|u\|^{p_1^-}- C \alpha_k\|u\|^{r^{+}},
		\end{align}
where $C,C_1,C_2  $ are some positive constants.

Define the  function $\mathcal G:\RR\to \RR,$
	$$\mathcal G(t)= \frac{1}{2}t^{  p_1^-}- C \alpha_k t^{r^+}.$$  Then it can be derived by a simple computation that  $G$ attains its maximum at $$\delta_k={\left(\frac{ p_1^-}{2r^+{ C }\alpha_k  }\right)^{1/(r^+- p_1^-)}}$$ and the maximum value of $\mathcal G$ is
	{\begin{align}
		\mathcal G(\delta_k)&=\frac{1}{ 2}\left(\frac{ p_1^-}{2r^+{ C }\alpha_k  }\right)^{{p_1^-}/(r^+-p_1^-)}-{{ C }}\alpha_k\left(\frac{ p_1^-}{2r^+{ C }\alpha_k  }\right)^{{r^+}/(r^+- p_1^-)}\nonumber\\
		&=\left(\frac{1}{ 2}\right)^{r^+/(r^+- p_1^-)}\left(\frac{1}{{{ C }} \alpha_k}\right)^{{ p_1^-}/(r^+- p_1^-)}\left(\frac{ p_1^-}{r^+}\right)^{\theta p_1^-/(r^+-p_1^-)}\bigg(1-\frac{ p_1^-}{r^+}\bigg)\nonumber.
		\end{align}}
	Since $ p_1^-<r^+$ and $\alpha_k \to 0$ as $k \to +\infty,$ we obtain
	\begin{equation}\label{ft5}
		\mathcal{G}(\delta_k)\to +\infty\ \text{as}\ k\to +\infty.
	\end{equation}
	
Again,   \eqref{ft1} infers $\delta_k\to+\infty$ as $k\to +\infty.$ Thus for $u\in Y_k$ with $\|u\|=\delta_k,$  taking into account \eqref{ft4} and \eqref{ft5}, it follows that as $k\to +\infty$
	$$b_k=\inf_{u\in Y_k,\|u\|=\delta_k} \mathcal{I}(u)\to +\infty.$$
	\item[$(\mathcal A_2):$]  Let us assume that the  assertion $(\mathcal {A}_2)$ does not hold  for some $k.$ So there exists a sequence $(u_n)_n\subset X_k$ such that
	\begin{align}\label{ft6}
	\|u_n\|\to +\infty\ \text{and}\ \mathcal{I}(u_n)\geq 0.
	\end{align}
	Let us take $w_n:=\frac{u_n}{\|u_n\|},$ then $w_n\in X$ and $\|w_n\|=1.$ Since $X_k$ is of finite dimension, there exists $w \in X_k\setminus\{0\}$ such that up to a subsequence, still denoted by $(w_n)_n,$ $w_n\to w$  strongly and $w_n(x)\to w(x)$
	a.e. $x\in\Om$ as $n \to +\infty. $
	If $w(x)\not=0$ then $|u_n(x)|\to+\infty$ as $n \to +\infty.$
	Similar to \eqref{C0}, it follows that for each $x\in\Om$,
		\begin{align}\label{ft6.0}
		\frac{F(x,u_n(x))}{|u_n(x)|^{p_2^+}}|w_n(x)|^{p_2^+}\to+\infty .
		\end{align}
		Hence, using \eqref{ft6} and applying Fatou's lemma, we have
		\begin{equation}\label{ft7}
			\frac{1}{\|u_n\|^{p_2^+}}\int_{\Om}F(x,u_n)dx=\int_{\Om}\frac{F(x,u_n)}{|u_n(x)|^{p_2^+}}|w_n(x)|^{p_2^+}dx\to+\infty\ \text{as}\  
			\end{equation}
	Since  $\|u_n\|>1$ for large $n\in\mathbb{N}$, from  Proposition \ref{norm-modular} and \eqref{ft7}, we obtain as $n\to+\infty$
		\begin{align*}
		\mathcal{I}(u_n)&\leq\rho(u_n)-\int_\Om F(x,u_n)\,dx\\
	&	\leq  \|u_n\|^{ p_2^+}-\int_\Om F(x,u_n)\,dx\\&=\bigg(1-\frac{1}{\|u_n\|^{ p_2^+}}\int_\Om F(x,u_n)\;dx\bigg)\|u_n\|^{ p_2^+}\to-\infty,
		\end{align*}
	\noindent a contradiction to \eqref{ft6}. Therefore, for sufficiently large $k\in\mathbb N,$  we can get $\varrho_k>\delta_k>0$  such that $(\mathcal {A}_2)$ holds for $u\in X_k$ with $\|u\|=\varrho_k$.
	\end{itemize}
\end{proof}

\section{Proof of Theorem \ref{dual-fount-sol}}

For proving Theorem \ref{dual-fount-sol}, we first recall the   Dual fountain theorem due to Bartsch and Willem (see \cite[Theorem 3.18]{willem}). Considering  Lemma \ref{ftlem} and using the reflexivity and separability  of the Banach space $X$   we can define $X_k$ and $Y_k$ appropriately.

\begin{definition}
For $c\in \mathbb R,$ we say that $\mathcal{I}$ satisfies the $(C)_{c}^{*}$ condition (with respect to $Y_{k}$) if any sequence $(u_k)_k$ in $X$ with $u_{k}\in Y_{k}$ such that
	$$
\mathcal{I}(u_{k})\to c \text {~~and~~} \|\mathcal{I}'_{|_{Y_{k}}}(u_{k})\|_{E^*}(1+\|u_{k}\|)\to 0, \text{~~as~~~} k\to+\infty
$$
 contains a subsequence converging to a critical point of $\mathcal{I},$ where $X^*$ is the dual of $X.$
\end{definition}

\begin{thm}[Dual fountain Theorem]\label{dual}
	Let  $\Phi\in C^1(E,\RR)$  satisfy $\Phi(-u)=\Phi(u).$ If for each $k\geq k_0$ there exist $\varrho_{k}>\delta_{k}>0$ such that
	\begin{itemize}
		\item[$(\mathcal B_{1})$] $a_{k}=\inf\{\Phi(u):u\in Z_{k},\,\,\|u\|_{E}=\varrho_{k}\}\geq 0;$
		\item[$(\mathcal B_{2})$] $b_{k}=\sup\{\Phi(u):u\in Y_{k},\,\,\|u\|_{E}=\delta_{k}\}< 0;$
		\item[$(\mathcal B_{3})$] $d_{k}=\inf\{\Phi(u):u\in Z_{k},\,\,\|u\|_{E}\leq \varrho_{k}\}\to 0$ as $k\to+\infty;$
		\item[$(\mathcal B_{4})$] $\Phi$ satisfies the $(C)_{c}^*$ condition for every $c\in [d_{k_0},0[.$
	\end{itemize}
	Then $\Phi$ has a sequence of negative critical values converging to $0$.
\end{thm}

\begin{remark} Note that,  in { \cite{willem}}, assuming that the energy functional associated to the problem satisfies $(PS)_c^*$  condition the Dual fountain theorem is obtained using Deformation theorem which is still valid under Cerami condition. Therefore, like many critical point theorems the Dual fountain theorem holds under $(C)_c^* $ condition.
\end{remark}

Next lemma is due to \cite[Lemma 3.2]{miyagaki}

\begin{lemma}\label{Cc'}
	Suppose that the hypotheses in Theorem \ref{dual-fount-sol} hold, then $\mathcal{I}$ satisfies the $(C)_{c}^{*}$ condition.
\end{lemma}

\begin{proof}[Proof of Theorem \ref{dual-fount-sol}]
	For the reflexive, separable Banach space $X,$ define $X_k$ and $Y_k$ as in \eqref{not}. From the assumptions we have that $\mathcal{I}$ is even and by Lemma \ref{Cc'} we get  that $\mathcal{I}$ satisfies Cerami condition $(C)_c^*$  for all $c\in \RR.$ Thus for proving Theorem \ref{dual-fount-sol} we are only left with verifying the conditions $(\mathcal B_1)$-$(\mathcal B_3).$\\
	\textit{ $(\mathcal B_1)$:} For all $u\in Y_k$ with $\|u\|<1,$ arguing similarly  as we did for obtaining \eqref{ft4}, we can derive

	\begin{align}\label{df1}
	\mathcal {I}(u)&\geq \frac{1}{2}\left[\rho_{p_1}(u)+\rho_{p_2}(u)\right] - \int_{\Omega}F(x,u)\,dx\nonumber\\
	&\geq \frac{1}{2}\rho(u)-\left[\int_{\Omega}|u|^{p_2^+}\,dx- C(1)\int_{\Omega}|u|^{r(x)}\,dx\right] \nonumber\\
	&\geq \frac{1}{2}\|u\|^{p_2^+}
	-\|u\|^{p_2^{+}}_{L^{p_2^+}(\Omega)}- C(1)\left\{\|u\|^{r^{-}}_{L^{r(\cdot)}(\Omega)}+\|u\|^{r^{+}}_{L^{r(\cdot)}(\Omega)}\right\} \nonumber\\
	&\geq\frac{1}{2}\|u\|^{p_2^+}
	- \alpha_k^{p_2^+} C_1\|u\|^{p_2^{+}}- C_2\{\alpha_k^{r^-}\|u\|^{r^{-}}+\alpha_k^{r^+}\|u\|^{r^{+}}\} \nonumber\\
	&\geq\frac{1}{2}\|u\|^{p_2^+}- C_4 \alpha_k\|u\|,
	\end{align}
	
	Let us choose $\varrho_k=\left(C_4 \alpha_k/{2}\right)^{1/{( p_2^+-1)}}.$ Since $ p_2^+>1,$  \eqref{al} infers that
\begin{equation}\label{df2}
	\varrho_k\to0 \text {~~as~} k\to+\infty.
\end{equation}
	Thus for $u\in Y_k$ with $\|u\|=\varrho_k$ and for sufficiently large $k\in \mathbb N,$ from \eqref{df1} we have $\mathcal{I}(u)\geq0.$ \\
	\textit{$(\mathcal B_2)$:}
	{Suppose assertion $(\mathcal B_2)$  does not hold true for some given $k\in\mathbb N.$ Then there exists a sequence $(v_n)_n$ in $ X_k$ such that
		\begin{align}\label{df6}
		\|v_n\|\to +\infty, ~~~~~\mathcal{I}(v_n)\geq 0.
		\end{align} }
	Now arguing in a similar way  as in the proof of assertion $(\mathcal A_2)$ of Theorem \ref{ft}, we obtain \eqref{ft6.0} and \eqref{ft7} which combining with  Lemma \ref{norm-modular}  imply that
	\begin{align*}
		\mathcal{I}(v_n)
		&\leq\rho(v_n)-\int_\Om F(x,v_n)\,dx\\
		&	\leq  {\|v_n\|}^{ p_2^+}-\int_\Om F(x,v_n)\,dx\\&=\bigg(1-\frac{1}{\|v_n\|^{ p_2^+}}\int_\Om F(x,v_n)\;dx\bigg)\|v_n\|^{ p_2^+}\to-\infty \text{\;\;as} 
		\end{align*}
 a contradiction to \eqref{df6}.
	So, there exists $k_0\in\mathbb N$ such that for all $k\geq k_0$  we have $1>\varrho_k>\delta_k>0$  such that for $u\in X_k$ with $\|u\|=\delta_k$ the condition  $(\mathcal B_2)$ holds true.\\
	\textit{$(\mathcal B_3)$:} Since $X_k \cap Y_k\not=\emptyset,$ we get  that
	$d_k\leq b_k<0.$ Now for $u\in Y_k$ with $\|u\|\leq \varrho_k$ by \eqref{df1}, we get
	{\begin{align*}
		\mathcal{I}(u)\geq -C_4\alpha_k\|u\|\geq -C_4\alpha_k\varrho_k.
		\end{align*} }Therefore, combining \eqref{al} and \eqref{df2}, we obtain 	
	$$
		d_k\geq -C_4\alpha_k\varrho_k\to 0 \text{~~ as~} k\to+\infty.
	$$
Since $d_k<0,$ it follows that $\lim_{k\to+\infty} d_k=0.$
	Thus the proof of the theorem is complete.
\end{proof}

\section{Proof of Theorem \ref{sym-infinite-sol}}

First we state the following $\mathbb Z_2$-symmetric version of mountain pass theorem due to \cite[Theorem 9.12]{sm}. Here again we want to mention that in \cite{sm} this theorem is proved using $(PS)$-condition, which can also be proved using $C$-condition.

\begin{thm}(Symmetric Mountain pass Theorem)\label{sm}: Let $E$ be a real infinite dimensional Banach space and $\Phi\in C^1(E,\mathbb R)$ be
	an even functional satisfying the $(C)_c$ condition. Also let $\Phi$ satisfy the following:
	\begin{itemize}
		\item [$(\mathcal {D}_1)$] $\Phi(0) = 0$ and there exist two constants $\nu,\mu >0$ such that $\Phi(u) \geq \mu$ for all $ u\in E$ with $\|u\|_E=\nu.$
		\item[$(\mathcal{ D}_2)$] for all finite dimensional subspaces $\widehat E\subset E$ there exists $\ol R=\ol R(\widehat{E}) > 0$ such that $\Phi(u) \leq 0$ for all $u \in \widehat E\setminus B_{\ol R}(\widehat E),$ where $B_{\ol R}(\widehat E)=\{u \in \widehat E : \|u\|_E\leq \ol R\}.$
	\end{itemize}
	Then $\Phi$ poses an unbounded sequence of critical
	values characterized by a minimax argument.
\end{thm}

\begin{proof}[Proof of Theorem \ref{sym-infinite-sol} ]
	From the hypotheses of the theorem it follows that $\mathcal I$	is even and we have $\mathcal I(0)=0.$ Now we will prove that $\mathcal I$ satisfies the assertions in Theorem \ref{sm}.
	\begin{itemize}
		\item[$(\mathcal D_1):$] It follows from Lemma \ref{geo}$(i)$.		\item[$(\mathcal D_2):$] To show this, first claim that for any finite dimensional subspace
		$Y$ of $X$ there exists $\ol R_0=\ol R_0(Y)$ such that $\mathcal I(u)<0$ for all $u\in E\setminus B_{\ol R_0} (Y),$ where $B_{\ol R_0}(Y)=\{u \in X: \|u\|\leq \ol R_0\}.$ Fix $u\in X,\;\|u\|=1.$ For $t>1$ using \eqref{g2} and Lemma \ref{norm-modular},
		we get
		\begin{align}\label{sm1}
		\mathcal I(tu)&\leq \rho(tu)-\int_\Om F(x,u)dx\nonumber\\&\leq t^{p_2^+} \rho(u)
		-t^{p_2^+} M\int_{\Omega}|u|^{p_2^+}\,dx+ |\Omega| C_M\nonumber\\
		&=t^{p_2^+} \left[1- M \|u\|_{L^{p_2^+}(\Om)}^{p_2^+}\,\right]+ |\Omega| C_M.
		\end{align}
Since $Y$ is finite dimensional, all norms are equivalent on $Y$, which infers that there exists some constant $C(Y)>0$ such that $C(Y)\|u\|\leq\|u\|_{L^{p_2^+}(\Om)}.$ Therefore, from \eqref{sm1}, we obtain

\begin{align*}
   \mathcal I(tu)&\leq t^{p_2^+} \left[1- M (C(Y))^{p_2^+} \|u\|^{p_2^+}\,\right]+ |\Omega| C_M\\
   &=t^{p_2^+}\left[1- M (C(Y))^{p_2^+} \,\right]+ |\Omega| C_M.
\end{align*}
Now by choosing $M$ sufficiently large such that $M>\frac{1}{(C(Y))^{p_2^+}},$ from the last relation we yields that
$$\mathcal I(u)\to-\infty\text{\;\;\; as\;\;} t\to+\infty.$$
 Hence, there exists $\ol R_0>0$ large enough such that $\mathcal I(u)<0$ for all
 $u\in X$ with $\|u\|=\ol{R}$ and $\ol R\geq\ol R_0$. Therefore, $\mathcal I$ verifies $(\mathcal D_2)$.
	\end{itemize}
\end{proof}

\section{Proof of Theorem \ref{clk-infinite-sol}}

 First, we recall a new variant of Clark's theorem (see \cite[Theorem 1.1]{clkk}).

\begin{thm}\label{clkk}
Let $E$ be a Banach space, $\Phi\in C^1(E,\mathbb R)$. Let $\Phi$ be even and $\Phi(0)=0$. Also assume $\Phi$ satisfies
the $(PS)$-condition and bounded from below. If for any $k\in\mathbb N,$ there exists a $k$-dimensional subspace $E^k$ of $E$ and $\beta_k>0$ such that  $\displaystyle\sup_{E^k\cap B_{\beta_k}}\Phi(u)<0,$ where $B_{\beta_k}=\{u\in E: \|u\|_E=\beta_k\}$, then at least one of the following conclusions holds:
\begin{itemize}
\item[$(\mathcal M_1)$]There exists a sequence of critical points $(u_k)_k$ satisfying $\Phi(u_k)<0$ for all $k$ and $\|u_k\|_E\to0$ as $k\to+\infty.$
\item[$(\mathcal M_2)$]  There exists $l>0$ such that for any $0<b<l$ there exists a critical point $u$ such that $\|u\|_E=b$ and $\Phi(u)=0.$
\end{itemize}
\end{thm}
The corresponding energy functional is given as
\begin{align*}\mathcal I(u)&= \frac{1}{2}\int_{\R^{2N}\setminus(\mathcal{C}\Omega)^2}\frac{|u(x)-u(y)|^{p_1(x,y)}}{p_1(x,y)|x-y|^{N+s(x,y)p_1(x,y)}}dxdy+
\int_{\Omega}\frac{1}{\overline{p}_1(x)}|u|^{\overline{p}_1(x)}dx\\
&+\frac{1}{2}\int_{\R^{2N}\setminus(\mathcal{C}\Omega)^2}\frac{|u(x)-u(y)|^{p_2(x,y)}}{p_2(x,y)|x-y|^{N+s(x,y)p_2(x,y)}}dxdy+  \int_{\Omega}\frac{1}{\overline{p}_2(x)}|u|^{\overline{p}_2(x)}dx\\
&+\int_{\mathcal{C}\Omega} \frac{\beta(x)|u|^{\overline{p}_1(x)}}{\overline{p}_1(x)}\,dx+\int_{\mathcal{C}\Omega} \frac{\beta(x)|u|^{\overline{p}_2(x)}}{\overline{p}_2(x)}\,dx\\
&\quad - \lambda\int_{\Omega}\frac{|u|^{q(x)}}{q(x)}dx- \int_{\Omega}\frac{|u|^{r(x)}}{r(x)}dx.
\end{align*}
Next, we will prove the following lemma:

\begin{lemma}\label{psc}
Suppose the hypotheses in Theorem \ref{clk-infinite-sol} hold. Then $\mathcal I$ satisfies $(PS)_c$ for any $c\in\mathbb R.$
\end{lemma}

\begin{proof}
\noindent  Let $(v_n)_n$ be a sequence in $X$ such that \begin{align}\label{0.0}\mathcal I(v_n)\to c \text{\;\;\;and\;\;} \mathcal I'(v_n)\to 0 \text{\;in\;} X^*\text {\;\; as\;\;} n\to+\infty.\end{align}
Therefore,
\begin{align}\label{1}
\langle \mathcal I'(v_n), v_n-v_0\rangle \to 0 \text{\;\;as\;} n\to+\infty.
\end{align} Hence, we have that $(v_n)_n$ is bounded in $X.$ If not, then $v_n\to+\infty$ as $n\to+\infty.$ Using \eqref{0.0} and \eqref{1} and $(f_5),$ we deduce
\begin{align}\label{ps}
1+C+\|v_n\|&\geq \mathcal I(v_n)-\frac{1}{q^-}\langle \mathcal I'(v_n), v_n\rangle\nonumber\\
&\geq \frac {1}{2} \left[\rho(v_n )-\int_\Om F(x,v_n) dx-\frac {1}{q^-}\rho(v_n)+\frac {1}{q^-}\int_\Om f(x,v_n)v_n dx\right]\nonumber\\
&\geq \frac 12\left[\left(1-\frac{1}{q^-}\right) \|v_n\|^{p^-}+\left(\frac{1}{q^-}-\frac{1}{r^-}\right)\int_\Om |v_n|^{r(x)}dx\right]\nonumber\\&\geq \frac 12\left(1-\frac{1}{q^-}\right) \|v_n\|^{p^-},
\end{align}
which is a contradiction to the fact that  $v_n\to+\infty$ as $n\to+\infty$.
Now, since $X$ is reflexive, up to a subsequence, still denoted by $(v_n)_n$, we have $v_n\rightharpoonup v_0$ weakly as $n\to+\infty.$ Therefore,  as $n\to+\infty$ by Theorem \ref{embd-X}, arguing similar as in \eqref{C26}, we obtain
\begin{align}\label{0}v_n\to v_0 \text{\;\;strongly in\;} L^{\gamma}(\Om),\,1<\gamma(x)<p_s^*(x)\text{\;\; and\;\;} v_n(x)\to v_0(x) \text{\;a.e. in\;} \Om.\end{align}	
By $(f_5),$ H\"older inequality, \eqref{0}, boundedness of $(v_n)_n$ and Lemma \ref{lemA1}, arguing in a similar fashion as \eqref{C26}, we obtain  \begin{align}\label{2}
\left|\int_{\Om} f(x,v_n)(v_n-v_0)dx\right|\to 0\text{\;\; as\;} n\to+\infty.
\end{align}
 Hence, combining \eqref{1} and \eqref{2} and using the $(S_+)$ property of $\rho',$
we have $v_n\to v_0$ strongly in $X$ as $n\to+\infty.$

\end{proof}

\begin{proof}[Proof of Theorem \ref{clk-infinite-sol}:]
From the hypotheses we have that $\mathcal I$ is even and $\mathcal{I}(0)=0.$ Also Lemma \ref{psc} ensures that $\mathcal I$ satisfies $(PS)$-condition. But note that, $\mathcal I$ is not bounded from below on $X.$ Hence, we will use a truncation technique.
For that choose $h\in C^\infty([0,\infty),[0,1])$ such that
\begin{equation*}
h(t)=
\begin{cases}
	&1  \mbox{\;\;if}\ t\in [0,l_0] \\
	     &0  \mbox{\;\;if}\ t\in [l_1,\infty),
\end{cases}
\end{equation*}
where $l_0<l_1$ and set $\Psi(u):=h(\|u\|).$ Now we define the truncated functional $\mathcal J$ as:
\begin{align*}\mathcal J(u)&= \frac{1}{2}\int_{\R^{2N}\setminus(\mathcal{C}\Omega)^2}\frac{|u(x)-u(y)|^{p_1(x,y)}}{p_1(x,y)|x-y|^{N+s(x,y)p_1(x,y)}}dxdy+
\int_{\Omega}\frac{1}{\overline{p}_1(x)}|u|^{\overline{p}_1(x)}dx\\
&+\frac{1}{2}\int_{\R^{2N}\setminus(\mathcal{C}\Omega)^2}\frac{|u(x)-u(y)|^{p_2(x,y)}}{p_2(x,y)|x-y|^{N+s(x,y)p_2(x,y)}}dxdy+  \int_{\Omega}\frac{1}{\overline{p_2}(x)}|u|^{\overline{q}_2(x)}dx\\
&+\int_{\mathcal{C}\Omega} \frac{\beta(x)|u|^{\overline{p}_1(x)}}{\overline{p}_1(x)}\,dx+\int_{\mathcal{C}\Omega} \frac{\beta(x)|u|^{\overline{p}_2(x)}}{\overline{p}_2(x)}\,dx\\
&\quad -\Psi(u)\left( \lambda\int_{\Omega}\frac{|u|^{q(x)}}{q(x)}dx- \int_{\Omega}\frac{|u|^{r(x)}}{r(x)}dx\right).
\end{align*}
Then $\mathcal J\in C^1(X,\mathbb R)$ and $\mathcal J(0)=0.$ Also $\mathcal J$ is even.
Moreover from Lemma \ref{psc}, it can be shown that $\mathcal J$ satisfies $(PS)$-condition. Now we will show $\mathcal J$ is bounded from below. For $\|u\|>1,$ using  Lemma \ref{norm-modular}, we get
\begin{align*}
\mathcal J(u)\geq \frac 12 \rho(u)\geq \frac 12 \|u\|^{p_1^-}\to+\infty
\end{align*} as $\|u\|\to+\infty,$ that is  $\mathcal J(u)$ is coercive and hence bounded below on $X.$ Next, we claim that $\mathcal J$ verifies the assertion $(\mathcal M_1)$ of Theorem \ref{clkk}. For any $k\in\mathbb N$ and $0<R_k<l_0<1$ let us set $$\mathcal B_{R_k}=\{u\in X:\|u\|=R_k\}.$$ Also consider the $k$-dimensional subspace $X^k$ of $X$. Then for $u\in X^k\cap\mathcal B_{R_k}$ there are some constants $K_1,K_2>0$ such that
\begin{align}\label{ck1}
\mathcal J(u)&\leq\rho(u)-\int_{\Omega}F(x,u)dx\nonumber\\
&\leq \|u\|^{p_1^-}-\frac{\lambda}{q^+}\int_{\Omega}|u|^{q(x)}dx-\frac{1}{r^+}\int_{\Omega}|u|^{r(x)}dx\nonumber\\
&\leq\|u\|^{p_1^-}-K_1^{q^+}\frac{\lambda}{q^+}\|u\|^{q^+},
\end{align}
 since $X^k\cap\mathcal B_{R_k}$ being of finite dimension all norms on it are equivalent.
Now by letting $R_k\to 0$ as $k\to+\infty$  from \eqref{ck1}, we get $\sup_{X^k\cap \mathcal B_{\R_k}} \mathcal J(u)<0$ since $p_1^-<q^+.$ Furthermore,  for a given $u\in X$ from \eqref{ck1} it  follows that $\mathcal J(tu)<0$ for $t\to 0^+,$ that is $\mathcal J(u)\not=0.$ Thus $\mathcal J$ does not satisfy $(\mathcal M_2).$ Therefore, by appealing Theorem \ref{clkk}, there exists a sequence of critical points $(v_k)_k$ of $\mathcal J$ in $X$ such that $\|v_k\|\to0$ as $k\to+\infty.$ So, for $l_0>0$ there exists $\hat k_0\in\mathbb N$ such that for all $k\geq\hat k_0$
 we have $\|u\|<l_0$ which infers that $\mathcal J(u_k)=\mathcal I(u_k)$ for all $k>\hat k_0.$ Since the critical values of $\mathcal I$ are the solutions to \eqref{eq1}, the theorem follows.
\end{proof}

\end{document}